\documentclass[12pt,reqno]{amsart}
\usepackage[utf8]{inputenc}
\usepackage{graphicx}
\graphicspath{ {./images/} }
\usepackage{amsfonts,latexsym,amsthm,amssymb,amsmath,amscd,euscript,mathrsfs}

\usepackage{tikz}
\usepackage{tikz-3dplot}
\usetikzlibrary{arrows}
\tikzstyle{block}=[draw opacity=0.7,line width=1.4cm]

\definecolor{umichblue}{RGB}{0,39,76}
\definecolor{umichmaize}{RGB}{255,203,5}

\newtheorem{theorem}{Theorem}[section]

\newtheorem{lemma}{Lemma}[section]
\newtheorem{corollary}{Corollary}[section]

\theoremstyle{definition}

\newtheorem*{remark}{Remark}

\theoremstyle{remark}

\newcommand{\Z}{\mathbb{Z}}
\numberwithin{equation}{section}

\renewcommand{\epsilon}{\varepsilon}

\title{Constructions of Generalized MSTD Sets in Higher Dimensions}
\author{Elena Kim}\email{elena.kim@pomona.edu}
\address{Department of Mathematics, Pomona College
Claremont, CA 91711}

\author{Steven J. Miller}\email{sjm1@williams.edu, Steven.Miller.MC.96@aya.yale.edu}
\address{Department of Mathematics and Statistics, Williams College,
Williamstown, MA 01267}

\date{\today}

\thanks{We thank John Haviland for his help in the creation of  Theorem \ref{linear transformation} and his contribution of Figures \ref{tikzfringes} and  \ref{tikzmiddle}. Additionally, we thank John Lentfer and Fernando Trejos Suárez for their detailed feedback on drafts of this paper. Finally, we thank members of the SMALL 2020 REU at Williams College, especially Ph\'uc L\^am and those formerly mentioned, for enlightening conversations. This research was supported by NSF grants DMS1947438 and DMS1561945 and by Williams College.}

\begin{document}

\begin{abstract}

     Let $A$ be a set of finite integers, define $$A+A \ = \ \{a_1+a_2: a_1,a_2 \in
A\}, \ \ \ A-A \ = \ \{a_1-a_2: a_1,a_2 \in A\},$$ and for non-negative integers $s$ and $d$ define $$sA-dA\ =\ \underbrace{A+\cdots+A}_{s} -\underbrace{A-\cdots-A}_{d}.$$ A More Sums than Differences (MSTD) set is an $A$ where $|A+A| > |A-A|$. It was initially thought that the percentage of subsets of $[0,n]$ that are MSTD would go to zero  as $n$ approaches infinity as addition is commutative and subtraction is not. However, in a surprising 2006 result, Martin and O'Bryant proved that a positive percentage of sets are MSTD, although this percentage is extremely small, about $10^{-4}$ percent. This result was extended by Iyer, Lazarev, Miller, ans Zhang \cite{ILMZ} who showed that a positive percentage of sets are generalized MSTD sets, sets for $\{s_1,d_1\} \neq \{s_2, d_2\}$ and $s_1+d_1=s_2+d_2$ with $|s_1A-d_1A| > |s_2A-d_2A|$, and that in $d$-dimensions, a positive percentage of sets are MSTD. 

For many such results, establishing explicit MSTD sets in $1$-dimensions relies on the specific choice of the elements on the left and right fringes of the set to force certain differences to be missed while desired sums are attained. In higher dimensions, the geometry forces a more careful assessment of what elements have the same behavior as $1$-dimensional fringe elements. We study fringes in $d$-dimensions and use these to create new explicit constructions. We prove the existence of generalized MSTD sets in $d$-dimensions and the existence of $k$-generational sets, which are sets where $|cA+cA|>|cA-cA|$ for all $1\leq c \leq k$. We then prove that under certain conditions, there are no sets with $|kA+kA|>|kA-kA|$  for all $k \in \mathbb{N}.$
\end{abstract}

\maketitle
\setcounter{equation}{0}

\tableofcontents

\section{Introduction}
Given a finite set $A\subset\Z$, we define  the sumset $A+A$ and the difference set $A-A$ by
\begin{align} A+A &\ =\ \{a_1 + a_2 : a_1, a_2 \in A\}, \nonumber \\
A-A &\ =\ \{a_1 - a_2 : a_1, a_2 \in A\}.
\end{align}
It is natural to compare the sizes of $A+A$ and $A-A$ as we vary $A$ over a family of sets.  As addition is commutative while subtraction is not, a pair of distinct elements $a_1, a_2 \in A$ generates two differences $a_1-a_2$ and $a_2-a_1$ but only one sum $a_1+a_2$.  We thus expect that most of the time, the size of the difference set is greater than that of the sumset---that is, we expect most sets $A$ to be \emph{difference-dominant}.  A set whose sumset has the same number of elements as its difference set is called \emph{balanced}.  It is possible, however, to construct sets whose sumsets have more elements than their difference sets, sets known as \emph{sum-dominant} or \emph{More Sums Than Differences} (MSTD) sets.  The first example of an MSTD set was discovered by Conway in the 1960s: $\{0,2,3,4,7,11,12,14\}$. While there are numerous constructions of such sets and infinite families of such sets, one expects sum-dominant sets to be rare; however, in 2006, Martin and O'Bryant \cite{MO} proved that a positive percentage of sets are sum-dominant. They showed the percentage is at least $2 \cdot 10^{-7}$, which was improved by Zhao \cite{Zh2} to at least $4.28 \cdot 10^{-4}$ (Monte Carlo simulations suggest the true answer is about $4.5 \cdot 10^{-4}$). For a set $A \subset [0,n]$, one can note that each nonzero $k\in  A+A$ has about $n/4-|n-k|/4$ representations as a sum of two elements in $A$. This number of representations is large except when $k$ is close to $0$ or $2n$. In these cases, $k$ can only be made from a select few pairs of elements close to $0$ or $n$. Thus a clever choice of the \textquotedblleft fringe," elements near the ends of the set, determines the size of its sumset and difference set, not the choice of the middle of the set.\\

We can extend the idea of sumsets and difference sets by defining  
\begin{equation}sA-dA\ = \ \underbrace{A+\cdots+A}_{s} -\underbrace{A-\cdots-A}_{d}.
\end{equation}
Given $s_1, d_1, s_2, d_2$ with $s_1+d_1 = s_2+d_2 \geq 2$, a generalized MSTD set is a set of integers $A$ such that $|s_1A-d_1A| > |s_2A-d_2A|$. Let $\{x_{j}\}_{j=1}^k$,  $\{y_{j}\}_{j=1}^k$, $\{w_{j}\}_{j=1}^k$, and $\{z_{j}\}_{j=1}^k$ be sequences of non-negative integers, such that $x_{j}+y_{j}=w_{j}+z_{j}=j$ and $\{x_j,y_j\} \neq \{w_j,z_j\}$ for every $2\leq j\leq k$. We then define a chain of generalized MSTD sets as an $A$ such that $\left|x_{j}A-y_{j}A\right|>\left|w_{j}A-z_{j}A\right|$ for every $2\leq j\leq k$. A $k$-generational set is a set $A$ such that for a specific positive integer $k$ we have $A, A+A, \ldots, kA$ are all sum-dominant. Iyer, Lazarev, Miller, and Zhang \cite{ILMZ} proved, through explicit constructions, that  generalized MSTD sets, chains of generalized MSTD sets, and $k$-generational sets exist. The authors additionally showed that these sets represent a positive percentage of subsets of $[0,n]$ as $n \rightarrow \infty$. 

A natural question then arises: are such sets similarly common in higher-dimensional spaces? This paper seeks to answer this question as a sequel to \cite{ILMZ}. The previous work on higher-dimensional MSTD sets has been conducted by Do, Kulkarni, Miller, Moon, Wellens, and Wilcox \cite{DKMMWW} who looked at sets created by dilating $d$-dimensional polytopes with vertices in $\Z^d$. The authors then used a probabilistic argument that was similar to that of Martin and O'Bryant, not explicit constructions, to prove a positive percentage of MSTD sets in $d$-dimensions. In this paper, we extend ideas from \cite{ILMZ} and \cite{DKMMWW} to look at generalized  MSTD sets, chains of generalized MSTD sets, and $k$-generational sets in higher dimensions through explicit constructions. Interesting new features and complications arise in higher dimensions. Whereas on the line it is natural to consider subsets of the integers in a growing interval, in $d$-dimensions we construct our sets as subsets of various parallelograms. Additionally, higher dimensions requires a clever generalization of the fringes.

In Theorem \ref{generalized MSTD}, Theorem  \ref{chain of gen MSTD}, and Theorem  \ref{k generational}, we prove, respectively, the existence of $2$-dimensional generalized MSTD sets, chains of generalized MSTD sets, and $k$-generational sets with $2$-dimensional generalizations of the sets in \cite{ILMZ}.

In Lemma \ref{growth of kA}, we show that for a 2-dimensional set, $A$, with certain elements and for sufficiently large  $N$, the amount of missing sums in $kA$ for $k \geq N$ grows linearly.


The proof of this lemma examines elements, $(x,y)$ in $kA$, by using a $1$-dimensional result by Nathanson \cite{Na1} on $x$ and $y$. We can then determine whether $(x,y)$ is or is not in $kA$. We can then use this lemma to prove Theorem \ref{no k generational} that these sets cannot be $k$-generational for every $k,$ or in other words cannot have $|kA+kA|>|kA-kA|$ for all $k \in \mathbb{N}.$

We then extend our $2$-dimensional work to $d$-dimensions by extending our explicit constructions for generalized MSTD sets into $d$-dimensions. As our work builds on that of \cite{ILMZ}, we end with an appendix which corrects some mistakes in that paper.

\section{Proof of Existence of $2$-Dimensional Generalized MSTD Sets}

We begin with a brief survey of the proofs for $1$-dimensional results from \cite{ILMZ} that we extend to  $2$-dimensions. To prove the existence of generalized MSTD sets and chains of generalized MSTD sets, the authors construct explicit sets with these properties. See Figure \ref{fig:1-dimensional addition} and Figure \ref{fig:1-dimensional subtraction}, which we use with permission from the authors, for depictions of these sets. Their sets have nearly symmetric left and right fringes, $L$ and $R$, with $R$ slightly longer than $L$. Specifically these fringes are defined as:
\begin{equation*}
\begin{split}
  L &\ = \  \{0,1,3,4,\ldots,k-1,k,k+1,2k+1\}\\
 &\ =  [0,2k+1]\backslash\left(\{2\}\cup[k+2,2k]\right) 
\end{split}
\end{equation*}

\begin{equation}
\begin{split}
R &\ = \ \{0,1,2,4,5,\ldots,k,k+1,k+2,2k+2\}\\
 &\ = \ [0,2k+2]\backslash\left(\{3\}\cup[k+3,2k+1]\right).
 \end{split}
 \end{equation}
 
The authors then create a filled-in middle interval $M$ for their set. Let $n>4(2k^2+1).$

\noindent When $s_1>s_2$, the middle is 
\begin{equation}M \ = \ [2k^2+1-d_1,n-(2k^2+1-d_1)],
\end{equation} and when $s_2>s_1,$ the middle is
\begin{equation}M \ =\ [2k^2+1-s_1,n-(2k^2+1-s_1)].\end{equation}

\noindent Then the complete set is \begin{equation}A \ = \ L \ \cup \ M  \ \cup \ (n-R).\end{equation}

As $A$  is added and subtracted, the middle section and the fringes begin to overlap. This overlap controls the amount of elements in $sA-dA.$

\begin{figure}[h!]
    \centering
    \includegraphics[scale=.9]{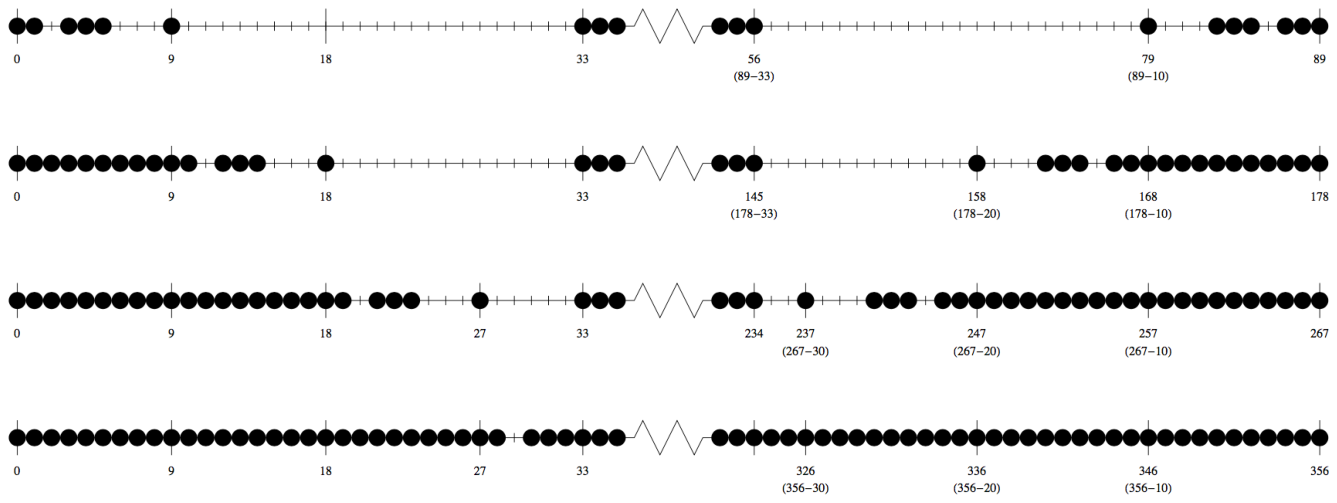}
    \caption{$A$, $A+A$, $A+A+A$, and $A+A+A+A$. The saw tooth means all elements are present in that range.}
    \label{fig:1-dimensional addition}
\end{figure}

\begin{figure}[h!]
    \centering
    \includegraphics[scale=.9]{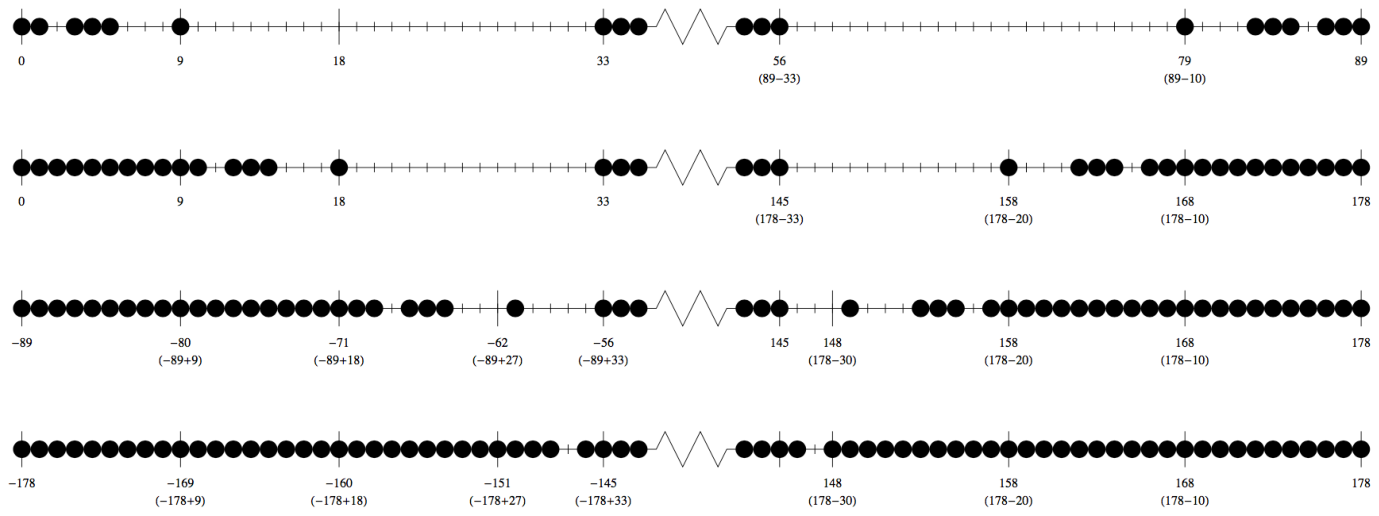}
    \caption{$A$, $A+A$, $A+A-A$, and $A+A-A-A$. The saw tooth means all elements are present in that range.}
    \label{fig:1-dimensional subtraction}
\end{figure}

We then adapt the $1$-dimensional set to prove the following theorem.

\begin{theorem}\label{generalized MSTD}
Let $s_1, d_1, s_2, d_2$ be non-negative integers such that $\{s_1, d_1\} \neq \{s_2,d_2\}$ and $s_1+d_1=s_2+d_2 \geq 2.$ Then there exists a finite, non-empty set $A$ in a sufficiently large $n \times n$ integer lattice such that $|s_1A-d_1A|> |s_2A-d_2A|.$
\end{theorem}

Let $k= s_1+d_1=s_2+d_2.$
We begin by defining two fringe sets $B_1$ and $B_2,$ where $B_2$ is slightly larger than $B_1$: 

\begin{equation*}B_1 \ = \ ([0, 2k+1] \times [0, 2k+1]) \ \setminus [\{(2,0), (0,2)\} \cup \{ (x,0), (0,y) : k+2\leq x, y \leq 2k \}]\end{equation*}
\begin{equation}B_2 \ = \ ([0, 2k+2] \times [0, 2k+2])  \ \setminus [\{(3,0), (0,3)\} \cup \{ (x,0),(0,y) : k+3\leq x,y \leq 2k+1 \}].\end{equation}\\

\begin{figure}[h]
\includegraphics[scale=.75]{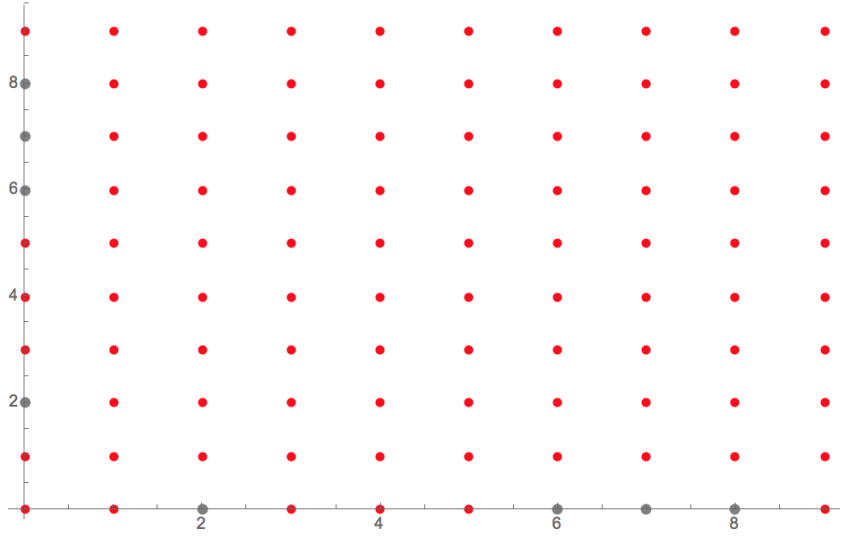}
\caption{The $B_1$ fringe set for $k=4$.}
\label{fig:B11}
\centering
\end{figure}

We now must prove a lemma characterizing the behavior of these fringe sets. We claim that for any $a, b \in \mathbb{N},$ $aB_1+bB_2$ has the same structure as $B_1$ and $B_2.$  The proof of this lemma relies on an analogous lemma in $1$-dimensions.

\begin{lemma}[ILMZ]\label{1D fringe}
Let $L=[0, 2k+1] \ \setminus (\{2\} \cup [k+2, 2k])$ and $R=[0, 2k+2]\ \setminus(\{3\}\cup [k+3, 2k+1])$. Then for all $a,b \in \mathbb{N},$ we have \begin{equation}
aL+bR \ = \ [0,2k(a+b)+(a+2b)]\ \setminus\end{equation}
$$\left(\{2k(a+b-1)+(a+2b+1)\} \cup [k(2a+2b-1)+(a+2b+1), 2k(a+b)+(a+2b-1)])\right)$$

\end{lemma}

This lemma is proved by using double induction, first on $a$, then on $b$. We can then prove our $2$-dimensional version of this lemma.

\begin{lemma}\label{2D fringe}
For all $a, b \in \mathbb{N},$ $aB_1+bB_2$, we have
$$aB_1 +bB_2 \ = \ \left([0,  2k(a+b)+(a+2b)] \times [0,  2k(a+b)+(a+2b)]\right) \ \ \setminus$$
$$[\{(2k(a+b-1)+(a+2b+1), 0), (0, 2k(a+b-1)+(a+2b+1))\} \ \cup$$
\begin{equation}\{(x,0), (0,y): k(2a+2b-1)+(a+2b+1) \leq x, y \leq 2k(a+b)+(a+2b-1)\}].\end{equation}
\end{lemma}

\begin{proof}
One can note that our $B_1$ and $B_2$ have  copies of the $1$-dimensional fringe sets from \cite{ILMZ} on their left and bottom edges. Addition between points of the form $(x,0)$ only results in sums of the same form. Similarly, addition between points of the form $(0,y)$ only results in sums of the same form. Adding of a point of the form $(x,0)$ to a point of the form $(0,y)$ only creates a point that already was in the fringe set. Thus we can use Lemma \ref{1D fringe} on the edges to prove our lemma.
\end{proof}

We then use the fringes to construct a set $A$ that lives in a sufficiently large $n \times n$ lattice such that $|s_1A-d_1A|>|s_2A-d_2A|.$ Let $n>4(2k^2+1).$ 

In order for our fringe sets to have the proper orientation in our set, we have to slightly modify the fringe definitions. Let $B_1=B_{1,1}$ and $B_2=B_{2,1}$ and let

 $$B_{1,2} \ = \ \left([0,2k+1] \times [0,2k+1]\right) \ \setminus$$
 \begin{equation*}[\{((2k+1)-2,0), (0,2)\} \ \cup \ \{ ((2k+1)-x,0), (2k+1,y) : k+2\leq x, y \leq 2k \}]\end{equation*}
 $$B_{2,2} \ = \ \left([0, 2k+2] \times [0, 2k+2]\right) \ \setminus$$
 \begin{equation}[\{(3,2k+2),(0,(2k+2)-3)\} \ \cup \ \{ (x,2k+2),(0,(2k+2)-y) : k+3\leq x,y \leq 2k+1 \}].\end{equation}

Since for any $s,d \in \mathbb{N}$, we have $|sA-dA|=|dA-sA|,$ we assume without loss of generality that $s_1 \geq d_1$ and $s_2 \geq d_2.$

We first examine the case where $s_1>s_2$, which implies that $d_2>d_1.$ We construct a set $M$ that fills the middle of our set.
\begin{align}
M \ =& \ \{(x,y) : 2k^2+1-d_1 \leq x \leq n-(2k^2+1-d_1), 0 \leq y \leq n\} \ \cup \nonumber \\
&\{(x,y) : 2k^2+1-d_1 \leq y \leq n-(2k^2+1-d_1), 0 \leq x \leq n\}.
\end{align}

We then define
\begin{equation}A \ = \ B_{1,1} \ \cup \ [(n,n)-B_{2,1}] \ \cup \ [(2k+1, n)-B_{1,2}] \ \cup \ [(n, 2k+2)-B_{2,2}].\end{equation}

In other words, $A$ is an $n \times n$ lattice where $B_1$ has been placed in the left corners, $B_2$ has been placed in the right corners, and $M$ is a filled-in cross that goes in between (but doesn't touch) the fringes. This can be seen in Figure \ref{fig:generalized MSTD}.

\begin{figure}
\includegraphics[scale=.8]{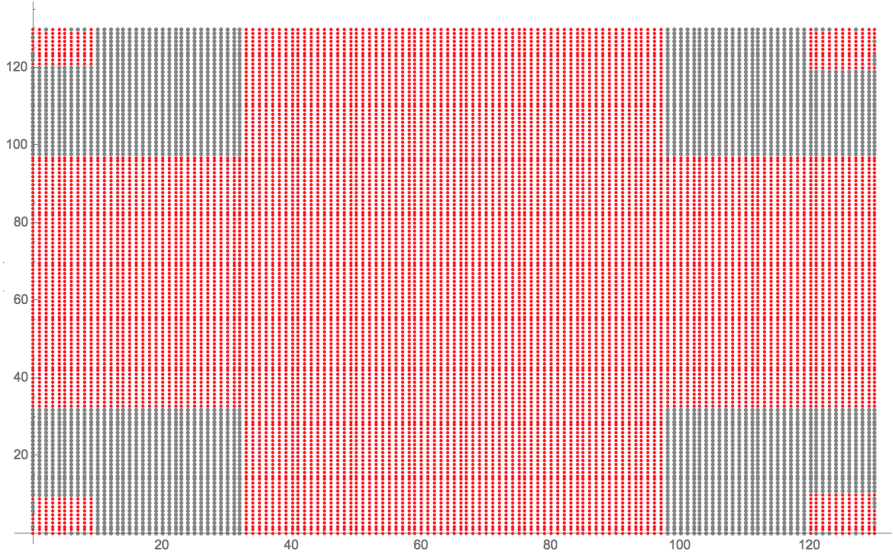}
\caption{The corresponding generalized MSTD set for $k=4$, $n=130$ and $s_1=4$, $d_1=0$ and $s_2=2$, $d_2=2$. }
\label{fig:generalized MSTD}
\centering
\end{figure}

One can see that as $M$ is the union of filled-in rectangles that we have $|s_1M-d_1M|=|s_2M-d_2M|$.
\\

We then examine the fringes. We begin with the left two fringes of $s_1A-d_1A$. These are identical up to rotation, so without loss of generality, we just examine one.  The portion which does not lie in $M$ is (up to translation)  $\left(s_1B_1+d_1B_2\right) \cap ([0, 2k^2-d_1] \times [0, 2k^2-d_1]) .$ We then compare $2k^2-d_1$ to $2k(s_1+d_1-1)+(s_1+2d_1+1)$.

By Lemma \ref{2D fringe}, we know that the the smallest missing points in the fringe are $(x,0)$ and $(0,y)$ where $x,y=2k(s_1+d_1-1)+(s_1+2d_1+1)$ and the smallest points in the missing intervals are $(x,0)$ and $(0,y)$ where $x,y=k(2s_2+2d_2-1)+(s_1+2d_2+1)$. We then compare these points to $2k^2-d_1$, to see how many points are missing from the portion of the fringe that does not lie in $M$ for $s_1A-d_1A$. When $s_1 > d_1$, we have 
\begin{align}
 2k^2 -d_1 &\ \geq \  2k^2-s_1 +1 \nonumber \\
 &\ = \ 2k(s_1+d_1)-k+d_1 +1 \nonumber \\
 &\ = \ s_1(2k+1)+d_1(2k+2)-2k+1 \nonumber \\
 &\ = \ 2k(s_1+d_1-1)+(s_1+2d_1+1)
\intertext{and} \nonumber 
 2k^2 -d_1 &\ < \ 2k^2+d_1+1 \nonumber \\
 &\ = \ k(2k-1)+(k+d_1+1) \nonumber \\
 &\ = \ k(2s_2+2d_2-1)+(s_1+2d_2+1).
\end{align}

Therefore, we know that $$\left(s_1B_1+d_1B_2\right) \cap  ([0, 2k^2-d_1] \times [0, 2k^2-d_1]) \ = \ ([0, 2k^2-d_1] \times [0, 2k^2-d_1]) \ \setminus$$ \begin{equation}((2k(s_1+d_1-1)+(s_1+2d_1+1), 0) \ \cup \ (0, 2k(s_1+d_1-1)+(s_1+2d_1+1))). \end{equation} Thus we are missing $4$ elements, $2$ in each fringe.

When $s_1 = d_1$, then we have 
\begin{align}
 2k^2 -d_1 &\ < \ 2k^2-s_1 +1 \nonumber \\
 &\ = \ 2k(s_1+d_1)-k+d_1 +1 \nonumber \\
 &\ = \ s_1(2k+1)+d_1(2k+2)-2k+1 \nonumber \\
 &\ = \ 2k(s_1+d_1-1)+(s_1+2d_1+1).
 \end{align}
We then know that \begin{equation}(s_1B_1+d_1B_2) \cap ([0, 2k^2 -d_1] \times [0, 2k^2 -d_1]) \ = \ [0, 2k^2 -d_1] \times [0, 2k^2 -d_1].\end{equation}Thus we are not missing any elements. 

We then examine the right two fringes of $s_1A-d_1A$. These are identical up to a rotation, so without loss of generality, we just examine one.  The portion which does not lie in  $M$ is  (up to translation) $(d_1B_1+s_1B_2) \cap ([0,2k^2-d_1] \times [0,2k^2-d_1]).$ Again by Lemma \ref{2D fringe}, we know that the the smallest missing points in the fringe are $(x,0)$ and $(0,y)$ where $x,y=2k(d_1+s_1-1)+(d_1+2s_1+1)$. Thus we then compare $2k(d_1+s_1-1)+(d_1+2s_1+1)$ to $2k^2-d_1$, to show that no points are missing from the portion of the fringe that does not lie in $M$. We have
\begin{align}
2k^2 -d_1 & \ < \ 2k^2 -d_1 +1 \nonumber \\
&\ = \ 2k^2-k+s_1 +1 \nonumber \\
&\ = \ 2k(d_1+s_1-1)+(d_1+2s_1+1).       
\end{align}

Therefore, we know that \begin{equation} (d_1B_1+s_1B_2) \cap ([0,2k^2-d_1] \times [0,2k^2-d_1]) \ = \ [0,2k^2-d_1] \times [0,2k^2-d_1]. \end{equation} Thus these fringes do not have any missing elements.

We then turn our attention to $s_2A-d_2A.$ We first  examine the  two left fringes of $s_2A-d_2A$. These are identical up to a rotation, so without loss of generality we just examine one.  The portion which does not lie in  $M$  is (up to translation) $(s_2B_1+d_2B_2) \cap ([0,2k^2-d_1] \times [0,2k^2-d_1]).$ By Lemma \ref{2D fringe}, we know that the the smallest missing points in the fringe are $(x,0)$ and $(0,y)$ where $x,y=2k(s_2+d_2-1)+(s_2+2d_2+1)$ and the smallest points in the missing intervals are $(x,0)$ and $(0,y)$ where $x,y=k(2s_2+2d_2-1)+(s_2+2d_2+1)$. We then compare these points to $2k^2-d_1$, to see how many points are missing from the portion of the fringe that does not lie in $M$. We note that $s_1>s_2\geq d_2>d_1.$ Therefore
\begin{align}
 2k^2 -d_1 & \ \geq \ 2k^2s_2 +1 \nonumber \\
 &\ = \ 2k(s_2+d_2)-k+d_2 +1 \nonumber \\
 &\ = \ s_2(2k+1)+d_2(2k+2)-2k+1 \nonumber \\
 &\ = \ 2k(s_2+d_2-1)+(s_2+2d_2+1)
 \intertext{and} \nonumber 
 2k^2 -d_1&\ < \ 2k^2+d_2+1 \nonumber \\
 &\ = \ k(2k-1)+(k+d_2+1) \nonumber \\
 &\ = \ k(2s_2+2d_2-1)+(s_2+2d_2+1).
\end{align}
Therefore, we know that 
$$(s_2B_1+d_2B_2) \cap ([0,2k^2-d_1] \times [0,2k^2-d_1]) \ = \ ([0,2k^2-d_1] \times [0,2k^2-d_1]) \ \setminus$$ 
\begin{equation}((2k(s_2+d_2-1)+(s_2+2d_2+1), 0) \ \cup \ (0,2k(s_2+d_2-1)+(s_2+2d_2+1))).\end{equation} Thus we are missing $4$ elements, $2$ in each fringe. 

We then, finally, examine the right two fringes of $s_2A-d_2A$. These are identical up to a rotation, so without loss of generality, we just examine one.  The portion which does not lie in $M$  is (up to translation) $(d_2B_1+s_2B_2) \cap ([0,2k^2-d_1] \times [0,2k^2-d_1]).$  By Lemma \ref{2D fringe}, we know that the the smallest missing points in the fringe are $(x,0)$ and $(0,y)$ where $x,y=2k(d_2+s_2-1)+(d_2+2s_2+1)$ and the smallest points in the missing intervals are $(x,0)$ and $(0,y)$ where $x,y=k(2d_2+2s_2-1)+(d_2+2s_2+1)$. We then compare these points to $2k^2-d_1$, to see how many points are missing from the portion of the fringe that does not lie in $M$. We have
\begin{align} 
2k^2 -d_1 & \ \geq \ 2k^2 -d_2 + 1 \nonumber \\
&\ =  \ 2k^2-k+s_2 +1 \nonumber \\
&\ = \ 2k(d_2+s_2-1)+(d_2+2s_2+1)
\intertext{and} \nonumber 
 2k^2 -d_1&\ < \ 2k^2+s_2+1 \nonumber \\
 &\ = \ k(2k-1)+(k+s_2+1) \nonumber \\
 &\ = \ k(2d_2+2s_2-1)+(d_2+2s_2+1).
\end{align}
 Therefore, we know that $$d_2B_1+s_2B_2 \cap ([0,2k^2-d_1] \times [0,2k^2-d_1]) \ = \ ([0,2k^2-d_1] \times [0,2k^2-d_1]) \ \setminus$$  
 \begin{equation}((2k(d_2+s_2-1)+(d_2+2s_2+1), 0) \ \cup \ (0, 2k(d_2+s_2-1)+(d_2+2s_2+1)).\end{equation} Therefore these two fringes have $4$ missing elements.\\

Thus we have shown for $s_1>s_2$  that $|s_1A-d_1A| > |s_2A-d_2A|.$\\

We then consider the case where $s_2>s_1$. We define $$M \ = \ \{(x,y) : 2k^2+1-s_1 \leq x \leq n-(2k^2+1-s_1), 0 \leq y \leq n\} \ \cup$$ 
\begin{equation}\{(x,y) : 2k^2+1-s_1 \leq y \leq n-(2k^2+1-d_1), 0 \leq x \leq n\}.\end{equation} Thus $A$ is constructed in the same manner as for $s_1>s_2$. The proof is  identical to show $|s_1A-d_1A| > |s_2A-d_2A|.$  \qed \\

We claim that we can create more constructions of generalized MSTD sets by applying injective linear transformations to the sets we just created. 

\begin{theorem}\label{linear transformation}
All MSTD sets are preserved under injective linear transformations.
\end{theorem}

\begin{proof}
We first show that the size of the sumsets and difference sets are preserved under injective linear transformation. Let $A \subseteq \mathbb{Z}^d$ and $T : \mathbb{Z}^d \to \mathbb{Z}^d$ be an injective linear transformation.
As $T$ in injective, we know that $|A+A|=|T(A+A)|$ and $|A-A|=|T(A-A)|.$ Let $a_1, a_2 \in A.$ As $T$ is a linear transformation, we have $T(a_1+a_2)=T(a_1)+T(a_2)$ and $T(a_1-a_2)=T(a_1)-T(a_2)$. Thus $|T(A+A)|=|TA+TA|$ and $|T(A-A)|=|TA-TA|$. We then conclude that $|A+A|=|TA+TA|$ and $|A-A|=|TA-TA|$. Induction can show that the the size of $sA+dA$ is preserved under $T.$
\end{proof}

\section{$2$-Dimensional Chains of Generalized MSTD Sets}

In this section we first  prove the following theorem on the existence of chains of generalized MSTD sets. We use the same sets that we constructed in the proof of Theorem \ref{generalized MSTD} in the previous section.

\begin{theorem}\label{chain of gen MSTD}
Let $\{x_{j}\}_{j=1}^k$,  $\{y_{j}\}_{j=1}^k$, $\{w_{j}\}_{j=1}^k$, and $\{z_{j}\}_{j=1}^k$ be finite sequences of non-negative integers of length $k$ such that $x_j +y_j=w_j +y_j=j,$ and $\{x_j, y_j\} \neq \{w_j, z_j\}$ for every $2 \leq j \leq k.$ There exists a  $2$-dimensional set A that satisfies $|x_jA-y_jA| >|w_jA-x_jA|$ for every $2 \leq j \leq k.$
\end{theorem}


To do this we first prove a series of lemmas. 

\begin{lemma}\label{same size sets}
For the set $A$ constructed in the previous section, we have $|s_1A-d_1A|=|s_2A-d_2A|$ for any $s_1+d_1=s_2+d_2$ such that $s_1+d_1 \neq k.$

\end{lemma}

\begin{proof}
As $M$ is the union of two filled-in rectangles, we first note that $s_1M-d_1M=s_2M-d_2M$. Thus for this proof we only consider the fringes. Since $|sA-dA|=|dA-sA|,$ we suppose without loss of generality that $s_1\geq d_1$.

We  begin with the case where $s_1+d_1=s_2+d_2 > k$ or in other words $s_1+d_1=s_2+d_2=k+c,$ where $c\geq 1$

We note for the left fringes in $s_1A-d_1A,$  rotating and translating them to lie in the bottom left corner, the two empty points on the edges closest to the corner have the coordinates 
$ (2k(s_1+d_1-1)+(s_1+2d_1+1), 0)$ and $(0, 2k(s_1+d_1-1)+(s_1+2d_1+1))$.  For the right fringes these points are $ (2k(d_1+s_1-1)+(d_1+2s_1+1), 0)$ and $(0, 2k(d_1+s_1-1)+(d_1+2s_1+1)).$

Similarly, we note for the left fringes in $s_2A-d_2A$,  rotating and translating them to lie in the bottom left corner, the two empty points on the edges closest to the corner have the coordinates 
$ (2k(s_2+d_2-1)+(s_2+2d_2+1), 0)$ and $(0, 2k(s_2+d_2-1)+(s_2+2d_2+1))$.  For the right fringes these points are $ (2k(d_2+s_2-1)+(d_2+2s_2+1), 0)$ and $(0, 2k(d_2+s_2-1)+(d_2+2s_2+1)).$

In the cases where $s_1> s_2$, recall that  $$M \ = \ \{(x,y) : 2k^2+1-d_1 \leq x \leq n-(2k^2+1-d_1), 0 \leq y \leq n\} \ \cup$$ 
\begin{equation}\{(x,y) : 2k^2+1-d_1 \leq y \leq n-(2k^2+1-d_1), 0 \leq x \leq n\}.\end{equation} We then compare $2k^2+1-d_1$, the smallest $x$ or $y$ can be for $(x,0) \in M$ or $(0,y) \in M$, with $2k(s_1+d_1-1) +(s_1+2d_1+1)$, the $x$ or $y$ coordinate for the smallest missing point in the left fringe of $s_1A-d_1A$:

\begin{align}
2k(s_1+d_1-1) +(s_1+2d_1+1) &\ = \ 2k(k+c-1)+(k+c+d_1+1) \nonumber \\
&\ > \ 2k^2+(k+d_1+1) \nonumber \\
&\ > \ 2k^2+1-d_1.
\end{align}

Thus the left fringe of $s_1A-d_1A$ is not missing any points.
This same inequality can be used for $(2k(d_1+s_1-1)+(d_1+2s_1+1)), (2k(s_2+d_2-1)+(s_2+2d_2+1))$, and  $(2k(d_2+s_2-1)+(d_2+2s_2+1)). $ Therefore the fringes are not missing any points. 

In the case where $s_2>s_1$ and thus 
\begin{align}
M \ =& \ \{(x,y) : 2k^2+1-s_1 \leq x \leq n-(2k^2+1-s_1), 0 \leq y \leq n\} \ \cup \nonumber \\
&\{(x,y) : 2k^2+1-s_1 \leq y \leq n-(2k^2+1-d_1), 0 \leq x \leq n\},\end{align} the same argument shows that the fringes are not missing any points.


We then examine the case where $s_1+d_1=c<k$. We show that the fringes do not intersect the middle. Up to translation, the end points for the left fringes of $s_1A-d_1A$ are $(2k(s_1+d_1)+(s_1+2d_1), 0)$and $(0, 2k(s_1+d_1)+(s_1+2d_1))$, the end points for the right fringes of $s_1A-d_1A$ are  $(2k(d_1+s_1)+(d_1+2s_1),0)$ and $(0,2k(d_1+s_1)+(d_1+2s_1))$, the end points for the left fringes of $s_2A-d_2A$ are $(2k(s_2+d_2)+(s_2+2d_2),0)$ and $(0, 2k(s_2+d_2)+(s_2+2d_2)),$ and finally the end points for the right fringes of $s_2A-d_2A$ are $(2k(d_2+s_2)+(d_2+2s_2),0)$ and $(0,2k(d_2+s_2)+(d_2+2s_2)).$  

As the argument for all of these points is the same, without loss of generality, we just examine the point $(2k(s_1+d_1)+(s_1+2d_1), 0)$.

In the case where $s_1>s_2,$ the  point in the middle section that is closest to $(2k(s_1+d_1)+(s_1+2d_1), 0)$  is $(2k^2+1-d_1,0)$. 
We note 

\begin{align}
2k(s_1+d_1)+(s_1+2d_1)&\ = \ 2kc+c +d_1 \nonumber \\
&\ \leq \ 2k(k-1)+(k-1) +d_1 \nonumber \\
&\ \leq \ 2k^2 -k -1 -d_1 \nonumber \\
&\ = \ 2k^2 -s_1 -1 \nonumber \\
&\ \leq \ 2k^2 -d_1 -1 \nonumber \\
&\ \leq \ 2k^2+1-d_1.
\end{align}

Thus the fringes do not intersect the middle. The argument for the the middle section when $s_2>s_1$ is practically identical. 

As $|s_1M-d_1M|=|s_2M-d_2M|$, it now suffices to show that \begin{equation}2|s_1B_1-d_1B_2| +2|d_1B_1-s_1B_2| \ = \ 2|s_2B_1-d_2B_2| +2|d_2B_1-s_2B_2|. \end{equation}

We have that $|s_1B_1-d_1B_2|=4k(s_1+d_1)-2k$, $|d_1B_1-s_1B_2|=4k(d_1+s_1)-2k$, $|s_2B_1-d_2B_2|=4k(s_2+d_2)-2k$, and  $|d_2B_1-s_2B_2|=4k(d_2+s_2)-2k$. Therefore, the fringe sets are all the same size, which implies that $|s_1A-d_1A|=|s_2A-d_2A|$.
\end{proof}

The next lemma defines the base expansion method in $2$-dimensions and proves a property about sets created through base expansion.

\begin{lemma}\label{combining sets}
Fix a positive integer $k$. Let $A, B \subset \mathbb{N}_0^2$ and chose $m>k \cdot \max(\{a:(a,y) \text{ or } (x,a) \in A\}).$ Let $C=A+m\cdot B$ (where $m \cdot  B$ represents the usual scalar multiplication). Then $|sC-dC|=|sA-dA||sB-dB|$ whenever $s+d \leq k.$
\end{lemma}

\begin{proof}
By the definition of $C$, we have $|C|\leq |A||B|.$ We claim that each element of $C=A+mB$ can be written uniquely as some $(a_1, a_2) +m(b_1, b_2)$ for $(a_1, a_2) \in A$ and $(b_1, b_2) \in B$. If $(a_1, a_2) + m(b_1,b_2)=(a'_1, a'_2) + m(b'_1,b'_2)$ then we have $(a_1-a_1', a_2-a_2') = m(b_1'-b_1,b_2'-b_2)$. Since $m>k \cdot \max(\{a:(a,y) \text{ or } (x,a) \in A\})$, we must have $(a_1, a_2)=(a'_1, a'_2)$ and $(b_1, b_2)=(b'_1, b'_2).$ Therefore we know that $|C|=|A||B|.$

We know that $|C \pm C| = |A \pm A||B \pm B|$. We now claim each element of $C \pm C$ is uniquely written as $(\Tilde{a_1}, \Tilde{a_2}) \pm m(\Tilde{b_1}, \Tilde{b_2})$  where $(\Tilde{a_1}, \Tilde{a_2}) \in A \pm A$ and $(\Tilde{b_1}, \Tilde{b_2}) \in B \pm B$.  If $(\Tilde{a_1}, \Tilde{a_2}) \pm m(\Tilde{b_1}, \Tilde{b_2})=(\Tilde{a_1}', \Tilde{a_2}') \pm m(\Tilde{b_1}', \Tilde{b_2}'),$ then we have $(\Tilde{a_1}-\Tilde{a_1}', \Tilde{a_2}-\Tilde{a_2}') \mp m(\Tilde{b_1}'-\Tilde{b_1}, \Tilde{b_2}'-\Tilde{b_2}).$ Again, by the size of $m,$ this can only happen when $(\Tilde{a_1}, \Tilde{a_2})=(\Tilde{a_1}', \Tilde{a_2}')$ and when $(\Tilde{b_1}, \Tilde{b_2})=(\Tilde{b_1}', \Tilde{b_2}').$ Therefore, we know $|C\pm C| =|A \pm A||B \pm B|.$
\end{proof}

The next lemma is a further generalization.

\begin{lemma} \label{combining sets induction}
Fix a positive integer $k$. Say that $A_{1},\ldots,A_{k}\subset \mathbb{N}_0^2$. Choose some
$m>k\cdot \max(\{a:(a,y) \text{ or } (x,a) \in A_{i}\text{ for  }1\leq i \leq k \})$. Let $C=A_{1}+m\cdot A_{2}+\cdots+m^{k-1}\cdot A_{k}$ (where $m\cdot A_j$ is the usual scalar multiplication).
Then $\left|sC-dC\right|=\prod_{j=1}^{k}\left|sA_{j}-dA_{j}\right|$
whenever $s+d\leq k$.
\end{lemma}

\begin{proof}
This can be proved using induction on the previous lemma. 
\end{proof}

With these lemmas, we can now prove Theorem $3.1$ about the existence of chains of generalized MSTD sets. We first restate the theorem. \\

\noindent\emph{Theorem 3.1.}
Let $x_j, y_j, w_j, x_j$ be finite sequences of non-negative integers of length $k$ such that $x_j +y_j=w_j +y_j=j,$ and $\{x_j, y_j\} \neq \{w_j, z_j\}$ for every $2 \leq j \leq k.$ There exists a  $2$-dimensional set A that satisfies $|x_jA-y_jA| >|w_jA-x_jA|$ for every $2 \leq j \leq k.$

\begin{proof}

For each $i$, choose a set $A_{i}$ such that $\left|x_{i}A_{i}-y_{i}A_{i}\right|$ $>$ $\left|w_{i}A_{i}-z_{i}A_{i}\right|$,
and for $j \neq i$, $|x_{j}A_{i}$ $-$ $y_{j}A_{i}|$ $=$ $|w_{j}A_{i}$ $-$ $z_{j}A_{i}|$. We know such a set exists,
because of Theorem \ref{generalized MSTD} and Lemma \ref{same size sets}. Next, choose some $m>k\cdot\max(\{a:(a,y) \text{ or } (x,a) \in A_{k}\text{ for }1\leq i \leq k \})$.
Define $A=A_{1}+mA_{2}+m^{2}A_{3}+\cdots+m^{k-1}A_{k}$. We have that
for each $2\leq j\leq k$

\begin{align}
\left| x_{j}A-y_{j}A\right| &\ = \ \prod_{i=1}^{k}\left|x_{j}A_{i}-y_{j}A_{i} \right| \nonumber \\
 &\ = \ \left|x_{j}A_{j}-y_{j}A_{j}\right|\cdot\prod_{i\neq j}\left|x_{j}A_{i}-y_{j}A_{i}\right| \nonumber \\
 &\ = \ \left|x_{j}A_{j}-y_{j}A_{j}\right|\cdot\prod_{i\neq j}\left|w_{j}A_{i}-z_{j}A_{i}\right| \nonumber \\
 &\ > \ \left|w_{j}A_{j}-z_{j}A_{j}\right|\cdot\prod_{i\neq j}\left|w_{j}A_{i}-z_{j}A_{i}\right| \nonumber \\
 &\ = \ \left|w_{j}A-z_{j}A\right|.
 \end{align}
 
\end{proof}

A corollary on the existence of $k$-generational sets immediately follows from this theorem.

\begin{corollary} \label{k generational}
For each $k \in \mathbb{N}$ there exists a $k$-generational set. That is, for each $k,$ there exists a set $A$ such that $|cA+cA|>|cA-cA|$ for all $1 \leq c \leq k.$ 
\end{corollary}


This result begs the question: Are there any sets such that $|kA+kA|>|kA-kA|$ for all  $k \in \mathbb{N}$? We begin  to answer this question by proving  that sets with certain properties cannot have  $|kA+kA|>|kA-kA|$ for all  $k \in \mathbb{N}$.

We begin with a $1$-dimensional theorem proved by Nathanson \cite{Na1} that is used in our $2$-dimensional proof. 
\begin{theorem}[Nathanson]\label{Nathanson}
Let $A=\{a_0, a_1, \ldots, a_k\}$ be a finite set of integers with $a_0=0<a_1 < \cdots < a_m=a$ and $(a_1, a_2, \ldots, a_m)=1.$ Then there exists non-negative integers $c$ and $d$ and sets $C \subset [0, c-2]$ and $D \subset [0,d-2]$ such that for all $k \geq a^2m$ $$kA \ = \ C\cup [c, ka-d] \cup (ka-D).$$
\end{theorem}






We then use this theorem to prove a lemma stating that for certain $2$-dimensional sets, the amount of elements missing from $|kA|$ grows linearly. 

\begin{remark}
\cite{ILMZ} improved Theorem \ref{Nathanson} with a lower bound on $k$. However, we discovered their proof doesn't hold in total generality, a point we correct in Theorem \ref{correction}. Our end conclusion, that a set $A$ with specific properties cannot have $|kA+kA|>|kA-kA|$ for all  $k \in \mathbb{N}$ does not necessitate the lowest possible bound on $k$. Thus, as we know the classical result by Nathanson to be true, we choose to rely on it, instead of the similar theorem from \cite{ILMZ}.
\end{remark}

\begin{figure}
\includegraphics[scale=.8]{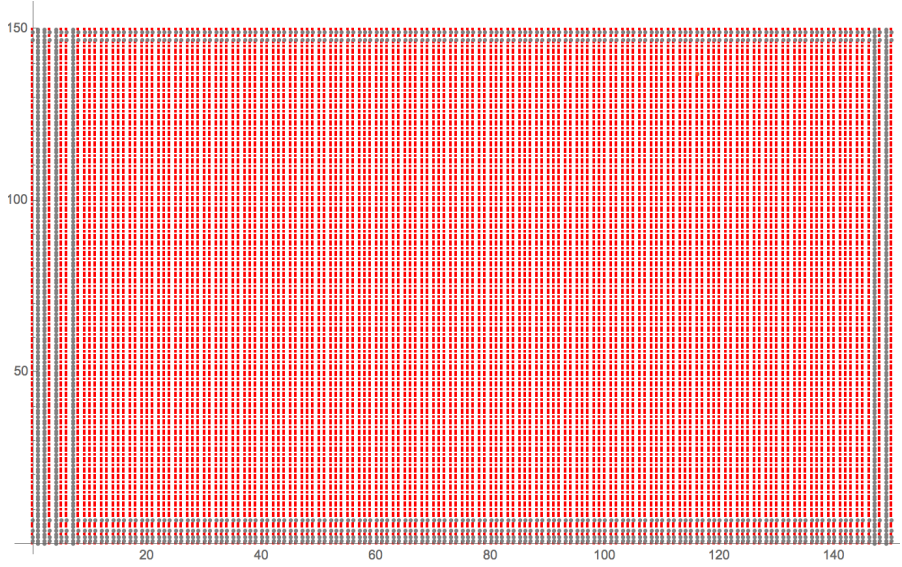}
\caption{$kA$ where $A=\{(0,0), (3,0), (0,3), (5,0), (0,5), (3,3), (5,3), (3,5), (5,5)\}$ and $k=30$}
\centering
\label{fig:stablized  set}
\end{figure}

\begin{lemma} \label{growth of kA}
Let $A=\{(a_1, b_1), (a_2, b_2), \ldots, (a_m, b_m)\}$ where $a_i, b_i$ are non-negative integers. Let $a$ be the smallest non-zero $a_i$, $a'$ be the largest $a_i$, $b$ be the smallest non-zero $b_i$ and $b'$ be the largest $b_i$, and $N=\max\{2a'^2, 2b'^2\}$. If $a$ and $a'$ are coprime,  $b$ and $b'$ are coprime, and 
$\{(0,0), (a,0), (0,b), (a',0),$ $(0,b'),(a,b),(a,b'), (a',b), (a',b') \} \subset A,$ then for $k \geq N$  and for some constants $\ell, \ell_1, \ell_2$, we have $|kA|=k^2a'b' -\ell-\ell_1a'k -\ell_2b'k$. 
\end{lemma}

\begin{remark}
Figure \ref{fig:stablized  set} demonstrates what these sets look like for $k \geq N$.
\end{remark}

\begin{proof}
 First, we use Nathanson's theorem to say that for $\mathcal{A}=\{0,a,a'\}$, there exists integers $c_1$ and $d_1$ and sets $C_1 \subset [0,c_1-2]$ and $D_1\subset [0,d_1-2]$ such that for $k \geq 2a'^2$, we have $k\mathcal{A}= C_1 \cup [c_1, ka'-d_1] \cup ka'-D_1.$ Additionally, for $B=\{0,b,b'\}$, we know there exists integers $c_2$ and $d_2$ and sets $C_2 \subset [0,c_2-2]$ and $D_1\subset [0,d_2-2]$ such that for $k \geq 2b'^2$, we have $kB= C_2 \cup [c_2, kb'-d_2] \cup (kb'-D_2).$

For this proof, we use the concept of fringes and a middle, however, they look different than in previous sections. The fringes are a frame around the set and and the middle is a rectangle in middle of the set. More specifically, we define the fringes to be 
\begin{equation}F \ = \ kA  \setminus M= kA \cap [\{ (x,y) : x \in C_1 \cup D_1\} \cup \{ (x,y) : y \in C_2 \cup D_2\}], \end{equation}

and we define the middle of $kA$ as 

\begin{equation}M \ = \ kA \cap \left[ \{(x,y) : x \in [c_1, ka'-d_1]\}] \cup [ \{(x,y) : y \in [c_2, kb'-d_2]\} \right].\end{equation} 

Let $N=\max\{2a'^2, 2b'^2\}.$ We first claim that $M$ is completely filled-in for $k \geq N$.  We rely on Nathanson's proof that showed for $A'=\{0, a, a'\}$ and for $k \geq 2a'^2$, the middle interval of $kA'$, $[c_1, ka'-d_1]$, is completely filled.  Thus as $\{(0,0), (a,0), (a',0)\} \subset A$, we know that $\{(x, 0): x \in [c_1, ka'-d_1] \} \subset k_aA$ where $k_a \geq 2a'^2$.  We can use the same argument to show that  $\{(0,y): y \in [c_2, kb'-d_2] \} \subset k_bA$ where $k_b \geq 2b'^2$. Let $k \geq N$.
Thus we know that for $x \in [c_1, ka'-d_1]$, $(x,0) \in kA$ and for  $y \in [c_2, kb'-d_2]$, $(0,y) \in kA$. We then examine $(x,y)=(x,0)+(0,y)$ where $x \in [c_1, ka'-d_1]$ and $y \in [c_2, kb'-d_2]$. We know that $(x,0)$ and $(0,y)$ can be made in most $k$ sums. As $\{(a,b),(a,b'),(a',b), (a'b')\} \subset A$, we know we can replace the elements $(a', 0) $, $(a, 0)$  and $(0, b')$, $(0, b)$ which are summed to create $(x,0)$ and $(0,y)$, respectively with the elements of $\{(a,b')$, $(a',b)$, $(a',b')$, $(a,b)\}$. Using this substitution, it is clear that the number of elements of $A$ needed to sum to create $(x,y)$ is at most $k.$ 
Thus, for $k\geq N$, the middle of $kA$ is not missing any elements. 

We then examine the fringes of $kA$ to show that they are always missing a set number of points, $C$, and an amount of points that linearly depends on $k$, $\ell_1ka'+ \ell_2kb'$. Suppose $(x, 0) \in kA$ for $k \geq N$ where $x \leq c_1-2$ or $x \geq ka'-(d_1-2)$. By Nathanson's theorem, we know that if $x \in (2a'^2)A$, then $x\in kA$ for $k > 2a'^2$. Similarly, if  $x \notin (2a'^2)A$, then $x\notin kA$ for $k > 2a'^2$. Thus $(x,0) \in kA$ for $k > 2a'^2$ if and only if $(x, 0) \in (2a'^2)A$. The same argument can be applied to $(0,y)$ for $y \leq c_2-2$ or $y \geq kb'-(d_2-2)$. We note that if $x \in (2a'^2)A$, for $x \leq c_1-2$ or $x \geq ka'-(d_1-2)$, then $x \in C_1\cup D_1$. Similarly,  if $y \in (2y'^2)A$ for $y \leq c_2-2$ or $y \geq kb'-(d_2-2)$, then $y \in C_2\cup D_2$.


We then have that for $(x,y)$, if  $x \in C_1 \cup D_1$ and $y \in C_2 \cup D_2$ then for $k \geq N$,  $(x,y)$ is always in $kA$ or never in $kA$. Let the total number of these missing elements be $\ell'$. If we have $x' \in C_1 \cup D_2$ and $y \notin C_2 \cup D_2$, then $x'$ is either always or never present and $y$ is always present as it is in the middle. If $x'$ is never present, then we are missing a column defined by $\{(x',y) : y \in [c_2,kb'-d_2]\}$. If $x'$ is present, then that column is in $kA.$ The length of  of this column is $kb'-d_2-c_2$. We can examine  $x \notin C_1 \cup D_1$ and $y \in C_2 \cup D_2$ to find that there may also be missing rows. The length of these rows is  $ka'-d_1-c_1.$ Let $\ell_1$ be the total number of missing columns and $\ell_2$ the missing rows. Then let $\ell=\ell' - \ell_1(d_1+c_1) -\ell_2(d_2+c_2)$. Then we note that the number of missing elements in $kA$ for $k \geq N$ is  $\ell +\ell_1ka' +\ell_2kb'.$
\end{proof}


We then use the previous lemma to prove a result about $k$-generational sets. 

\begin{lemma} \label{no k generational}
Let $A=\{(a_1, b_1), (a_2, b_2), \ldots, (a_m, b_m)\}$ where $a_i, b_i$ are non-negative integers. Let $a$ be the smallest non-zero $a_i$, $a'$ be the largest $a_i$, $b$ be the smallest non-zero $b_i$ and $b'$ be the largest $b_i$ and $N=\max\{2a'^2, 2b'^2\}$. If $a$ and $a'$ are coprime,  $b$ and $b'$ are coprime, and $\{(0,0), (a,0), (0,b), (a',0), (0,b'),(a,b'), (a',b),$ $(a',b'), (a,b) \} \subset A,$ then for $k \geq N$  we have $|kA-kA| \geq |kA+kA|$.

\begin{proof}
Let $k \geq N.$ We then know that $|kA|=k^2a'b' -\ell-\ell_1a'k -\ell_2b'k$, or in other words, $kA$ has a filled-in middle and a total of $\ell_1$ missing columns on left and right sides of the set and a total of $\ell_2$ missing rows on top and bottom edges of the set. We know  both $2kA$ and $kA-kA$  are  subsets of an integer lattice with $4k^2a'b'$ points and we note that $|2kA|=4k^2a'b' -\ell-\ell_1(2a'k) -\ell_2(2b'k)$. We then examine how many points $kA-kA$ is missing.  $kA$ is the union of a rectangular filled-in middle and filled-in rows and columns where rows and columns overlap in the corners of the set. The rectangular middle, the rows, and the columns are all preserved under subtraction. Thus $kA-kA$ has at least $4k^2a'b' -\ell-\ell_1(2a'k) -\ell_2(2b'k)$ points. However, under subtraction,  points from different rows and columns and the middle may interact to add new points to the set $kA-kA$. Therefore  $|kA-kA| \geq 4k^2a'b' -\ell-\ell_1(2a'k) -\ell_2(2b'k)$. We conclude that $|kA-kA| \geq |kA+kA|.$
\end{proof}

\end{lemma}

\section{Other $2$-Dimensional Constructions}

The square lattice construction that was used in the previous two sections can also be generalized to polygons that have integer vertices and are locally point symmetric. From the work of \cite{DKMMWW}, we know that the shape of locally point symmetric polygons is preserved under subtraction. The only shapes that satisfy these two conditions are parallelograms. Thus, we  generalize our construction to  rectangles, then to all parallelograms with integer vertices. We first qualitatively describe the sets, then explicitly write their formulas. Figure \ref{fig:parallelogram MSTD} demonstrates the shape of the sets. The fringe sets and middle sets for these sets mirror the construction of the square lattice case. The fringe sets are created by placing the $1$-dimensional fringes along the edges by each vertex and then filling in the rest of the corner. In the rectangle case, there are almost completely filled-in squares in each of the corners, and in the parallelogram case, the squares are appropriately sheared. We place a slightly larger fringe set in the right corners than the left corners. The middle set is created with a filled-in wide cross in the middle of the set, and sheared appropriately in the parallelogram.  \\

For a rectangle construction, we have side lengths $ n_1, n_2 \in \mathbb{N}$  where $n_1, n_2>4(2k^2+1).$ $B_{1,1}$, $B_{1,2}$, $B_{2,1}$, and $B_{2,2}$ are identically constructed as in the square lattice case. 
 In the case where $s_1> s_2,$
\begin{align}
M \ = \ \{(x,y) : 2k^2+1-d_1 \leq x \leq n_1-(2k^2+1-d_1), 0 \leq y \leq n_2\} \ \cup \nonumber \\
\{(x,y) : 2k^2+1-d_1 \leq y \leq n_2-(2k^2+1-d_1), 0 \leq x \leq n_1\}.\end{align}

 In the case where $s_2>s_1,$
\begin{align}
M \ =& \ \{(x,y) : 2k^2+1-s_1 \leq x \leq n_1-(2k^2+1-d_1), 0 \leq y \leq n_2\} \ \cup \nonumber \\ 
&\{(x,y) : 2k^2+1-s_1 \leq y \leq n_2-(2k^2+1-d_1), 0 \leq x \leq n_1\}.\end{align}

In both cases we then construct our rectangle, $A,$ as follows:
\begin{equation}A \ = \ B_{1,1} \ \cup \ [(n_1,n_2)-B_{2,1}] \ \cup \ [(2k+1, n_2)-B_{1,2}] \ \cup \ [(n_1, 2k+2)-B_{2,2}].\end{equation}

For the parallelogram case, let $n_1, n_2, n_3 \in \mathbb{N}.$ We let  $n_1$ represent the length of the bottom and top edges, $n_2$ the height, and $n_3$  the length from the bottom left vertex to where a perpendicular line dropped from top left vertex reaches the bottom. We need to make $n_1, n_2, n_3 $ large enough so that both sides have more than  $4(2k^2+1)$ integer points along them, so we pass to a parallelogram scaled  by $4(2k^2+1)$ and relabel. We note the slope of the diagonal edge is $n_2/n_3=m.$

\begin{figure}
    \centering
    \includegraphics[scale=.8]{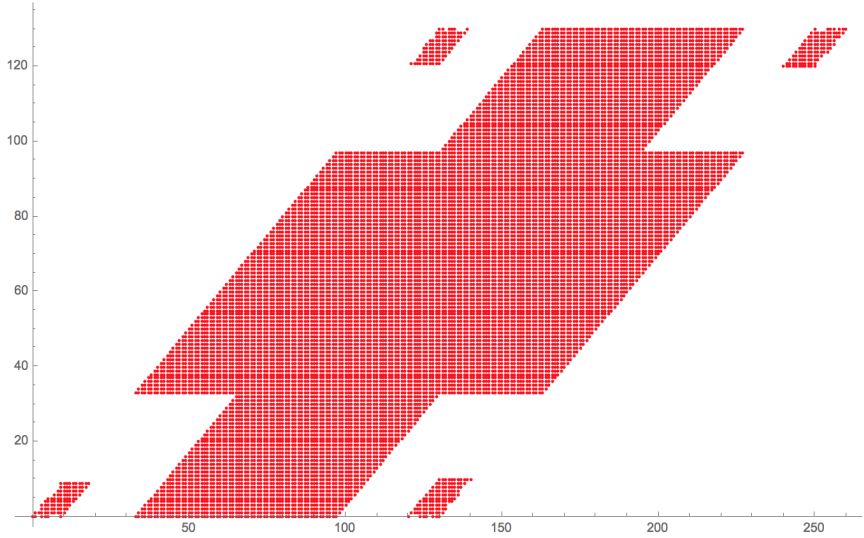}
    \caption{The generalized MSTD set for $k=4$, $n=130,$ $s_1=4$, $d_1=0$, $s_2=2$, and $d_2=2$ that has been sheared with slope $m=1$.}
    \label{fig:parallelogram MSTD}
\end{figure}

To create this set, we can use a mapping to shear the square and rectangle lattice sets that we already created. 

We define a map   $\varphi: \mathbb{Z}^2 \rightarrow \mathbb{Z}^2 $ by $\varphi(x,y)= (x+my, y).$ We know by Theorem \ref{linear transformation} that if $A$ is a set such that $|s_1A-d_1A|>|s_2A-d_1A|,$ $\varphi(A)$ has the same property. 
We can also explicitly describe the generalized MSTD sets that are subsets of parallelograms.  
We take the same $B_{1,1}, B_{1,2},B_{2,1}, B_{2,2}$ as before, but then we shear them:

$$B_{1,1} \ = \ \{(x+my,0),(0,y) : 0 \leq x, y \leq 2k+1 \}  \ \setminus$$ \begin{equation}[\{(2,0),(2m,2)\} \cup \{ (x,0), (my,y) : k+2\leq x, y \leq 2k \}],\end{equation}

$$B_{2,1} \ = \ \{(x+my,0),(0,y)  : 0 \leq x, y \leq 2k+2 \}  \ \setminus$$
\begin{equation}[\{(3,0), (3m,3)\} \cup\{ (x,0), (my,y) : k+3\leq x, y \leq 2k+1 \}],\end{equation}

$$B_{1,2} \ = \ \{(2k+1-(x+my),0),(2k+1,y)  : 0 \leq x, y \leq 2k+1 \} \ \setminus$$
$$[\{(2k+1-2,2k+1),(2k+1-2m,2)\} \ \cup$$ 
\begin{equation}  \{ (2k+1-x,2k+1),(2k+1-my,y) : k+2\leq x,y \leq 2k \}],\end{equation}

$$B_{2,2} \ = \ \{(x+my,2k+2),(0,2k+2-y)  : 0 \leq x, y \leq 2k+2 \}\ \setminus$$
$$[\{(3,2k+2),( 3m,2k+2-3)\} \ \cup$$ \begin{equation}\{ (x,2k+2),(my,2k+2-y) : k+3\leq x,y \leq 2k+1 \}].\end{equation}

We then similarly shear the middle set. In the case where $s_1>s_2$,
$$M \ = \ \{(x+my,y) : 2k^2+1-s_1 \leq x \leq n-(2k^2+1-d_1), 0 \leq y \leq n_2\} \ \cup$$ \begin{equation} \{(x+my,y) : 2k^2+1-s_1 \leq y \leq n-(2k^2+1-d_1), 0 \leq x \leq n_1\}.\end{equation}

For $s_2>s_1$.
$$M \ = \ \{(x+my,y) : 2k^2+1-s_1 \leq x \leq n-(2k^2+1-s_1), 0 \leq y \leq n_2\} \ \cup$$ \begin{equation}\{(x+my,y) : 2k^2+1-s_1 \leq y \leq n-(2k^2+1-s_1), 0 \leq x \leq n_1\}.\end{equation}

Finally, we define our complete set:
\begin{equation}A \ = \ B_{1,1}  \cup  [(n_1 +mm_2,n_2)-B_{2,1}]   \cup  [(2k+1 +mn_2, n_2)-B_{1,2}]  \cup  [(n_1 +m(2k+2), 2k+2)-B_{2,2}].\end{equation}

\section{$D$-Dimensional Constructions}
We now extend our rectangular construction of generalized MSTD sets to $d$-dimensions. These are $d$-dimensional rectangles with non-negative coordinates and one corner at $(0,0, \ldots, 0)$.  We first create our fringes, which are $d$-dimensional cubes with the $1$-dimensional fringes placed along their edges attached to the vertex, as seen in Figure \ref{tikzfringes}.

\begin{figure}
    \centering
    \includegraphics[scale=.8]{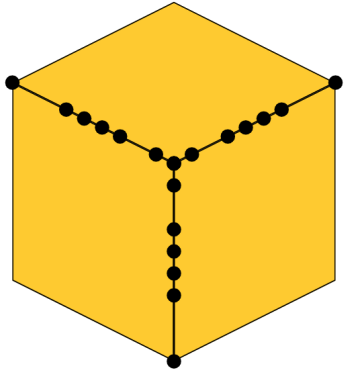}
        
        
        
        
    \caption{One of the fringe sets in $3$-dimensions. The black dots represent the included points on the edges. The rest of the cube is filled-in.}
    \label{tikzfringes}
\end{figure}

$$B_1 \ = \ \{(x_1,0, \ldots, 0), \ldots, (0,0, \ldots, x_d): 0 \leq x_1, \ldots, x_d \leq 2k+1\}\ \setminus$$
$$[\{(2,0, \ldots, 0), (0,2, \ldots, 0),  \ldots, (0,0, \ldots, 2)\} \ \cup$$
\begin{equation*} \{ (x_1, 0, \ldots, 0), (0, x_2, \ldots, 0), \ldots,  (0, 0, \ldots, x_d) : k+2 \leq x_1, x_2, \ldots x_d \leq 2k\}].  \end{equation*}

$$B_2 \ = \ \{(x_1,0, \ldots, 0), \ldots, (0,0, \ldots, x_d): 0 \leq x_1, \ldots, x_d \leq 2k+2\}\ \setminus$$
$$[ \{(3,0, \ldots, 0), (0,3, \ldots, 0),  \ldots, (0,0, \ldots, 3)\} \  \cup$$
\begin{equation} \{ (x_1, 0, \ldots, 0), (0, x_2, \ldots, 0), \ldots,  (0, 0, \ldots, x_d) : k+3 \leq x_1, x_2, \ldots x_d \leq 2k+1\}].  \end{equation}

However, to actually put these fringes into corners of our set, we need to change their orientation. We introduce new notation. Let the side lengths be $n_1, n_2, \ldots, n_d,$ where the side of length $n_j$ lies parallel to the $j$-th axis. We denote the fringes as $B_{i_1,i_2, \dots, i_d},$ where $i_j\in\{0,1\}$. If $i_j=0$,  then the $j$th coordinate of the corner $B_{i_1,i_2, \dots, i_d}$ is in is $n_j$.  If $i_j=1$,  then the $j$th coordinate of the corner $B_{i_1,i_2, \dots, i_d}$ lies in is $0$. We then define maps to create each fringe with proper orientation. The map $\varphi_{i_1,i_2, \ldots, i_d}$ corresponds to $B_{i_1,i_2, \dots, i_d}$. Let $\varphi_{1,i_2, \ldots, i_d}: \mathbb{Z}^d \rightarrow \mathbb{Z}^d$ be defined by
\begin{equation}\varphi_{1,i_2, \ldots, i_d} (x_1, x_2, \ldots, x_d) \ = \ (x_1, (2k+1)i_2-x_2, \ldots,(2k+1)i_d-x_d),\end{equation}

and let $\varphi_{0,i_2, \ldots, i_d}: \mathbb{Z}^d \rightarrow \mathbb{Z}^d$ be defined by
\begin{equation}\varphi_{0,i_2, \ldots, i_d} (x_1, x_2, \ldots, x_d)\ = \ (2k+2-x_1, (2k+2)i_2-x_2, \ldots,(2k+2)i_d-x_d). \end{equation}

Then we have $\varphi_{0,i_2, \ldots, i_d}(B_1)=B_{0,i_2, \dots, i_d}$ and $\varphi_{1,i_2, \ldots, i_d}(B_2)=B_{1,i_2, \dots, i_d}.$\\



We then define our middle section of the set. It is a $d$-dimensional rectangle with a $d$-dimensional cube missing  from each corner, as seen in Figure \ref{tikzmiddle}.

\begin{figure} 
    \centering
    \includegraphics[scale=.8]{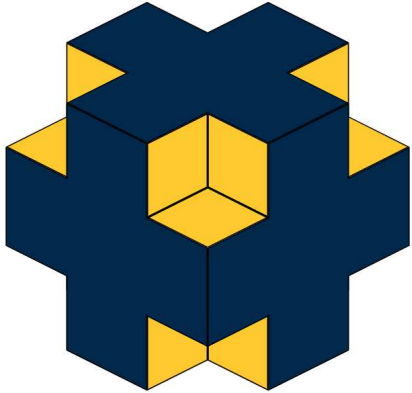}
    \caption{The shape of the middle set in $3$-dimensions.}
    \label{tikzmiddle}
\end{figure}

We define for $s_1>s_2$,
\begin{equation}M \ = \ \bigcup_{i=1}^d \{(x_1, x_2, \ldots, x_d): 2k^2+1-d_1 \leq x_j \leq n_j- (2k^2+1-d_1) \text{ for } j \neq i,  0 \leq x_i \leq n_i\}.\end{equation}


For $s_2>s_1$ we have,
\begin{equation} M \ = \ \bigcup_{i=1}^d \{(x_1, x_2, \ldots, x_d): 2k^2+1-s_1 \leq x_j \leq n_j- (2k^2+1-s_1) \text{ for } j \neq i,  0 \leq x_i \leq n_i\}.\end{equation}


\begin{equation} A \ = \ M \ \cup \end{equation}
$$\left[\bigcup_{i_2, \ldots, i_d \in \{0,1\}, i_1=0}((n_1, (1-i_2)n_2+i_2(2k+2), \ldots, (1-i_d)n_d+i_d(2k+2))-B_{0,i_2, \dots, i_d})\right] \ \cup$$
$$\left[\bigcup_{i_2, \ldots, i_d \in \{0,1\}, i_1=1}((2k+1, (1-i_2)n_2+i_2(2k+1), \ldots, (1-i_d)n_d+i_d(2k+1))-B_{1,i_2, \dots, i_d})\right].$$

We then can modify this construction to be a subset of a $d$-dimensional parallelogram by shearing. In the 2-dimensional case, only the edges parallel to the $y$-axis could be sheared in the $x$ direction. For the $3$-dimensional case, we are still able to shear the base, as in the $2$-dimensional case, but we can now also shear the edges parallel to the $z$-axis in the $x$ and the $y$ directions. Thus there are $3$ ways to shear the fringe. More generally, there are $d(d-1)/2$ positive directions that each of the edges can be sheared in, or in other words, there are $d(d-1)/2$ ways to shear the set. We therefore have $d(d-1)/2$ slopes $m_{1,2}, m_{1,2}, \ldots, m_{1,d}, m_{2,3}, \ldots, m_{2,d}, \ldots, m_{d-1,d}$ that define each way to shear the fringe. The subscripts refer to the coordinates involved in the shear, for example $m_{i,j}$ is the  $j$th axis sheared in the $i$th direction.

We then create a map that shears our rectangular constructions. We define $\psi: \mathbb{Z}^d \rightarrow \mathbb{Z}^d$ by $$\psi(x_1, x_2, \ldots, x_d) =$$ \begin{equation}(x_1 +m_{1,2}x_2 + m_{1,3}x_3 + \cdots + m_{1,d}x_d, x_2 +m_{2,3}x_3  + \cdots + m_{2,d}x_d, \cdots, x_d). \end{equation}

\section{Further Remarks}
We believe that there are several aspects of this paper that can be extended further in the future. First, work on $1$-dimensions in \cite{ILMZ} has shown positive percentages for generalized MSTD sets, chains of generalized MSTD sets, and $k$-generational sets using the methods of \cite{MO}. \cite{DKMMWW} extended the probabilistic methods of \cite{MO} into $d$-dimensions to find a positive percentage of $d$-dimensional MSTD sets. We believe that the same methods could be used to find positive percentages of generalized MSTD sets, chains of MSTD sets, and $k$-generational sets in $d$-dimensions.  

Second, we believe that the work on proving that for a set $A$ with specific properties, $|kA+kA|>|kA-kA|$ does not hold for all $k \in \mathbb{N}$ can be extended to sets with less restrictions. Consider slightly different conditions from those in Lemma \ref{growth of kA}:
Let $A=\{(a_1, b_1), (a_2, b_2), \ldots, (a_m, b_m)\}$ where $a_i, b_i$ are non-negative integers. Let $a$ be the smallest non-zero $a_i$, $a'$ be the largest $a_i$, $b$ be the smallest non-zero $b_i$ and $b'$ be the largest $b_i$. Suppose $a$ and $a'$ are coprime,  $b$ and $b'$ are coprime, and $\{(0,0), (a,0), (0,b),$ $(a',0), (0,b'), (a',b') \} \subset A.$ For this set of conditions where fewer elements are required to be in $A$, the growth of $|kA|$ is still linear, but the set is more complicated. Instead of having rows and columns missing in the fringes, the right and top parts of the fringes have a repeated jagged pattern. The width of the right fringe and the height of the top fringe remain constant as $k$ grows. Thus for large $k$, this set also has linear growth for $|kA|.$ Therefore, we believe it would be possible to come up with a similar lemma to Lemma \ref{growth of kA} with fewer conditions. In fact, we think that it would be possible to prove that for sufficiently large $k$, for any 2-dimensional set, the amount of the points missing from $kA$ has linear growth. A lemma about linear growth for any $2$-dimensional set $A$ would allow a general theorem that $|kA+kA|\leq |kA-kA|$ for large  $k$. We additionally hope that we can generalize a statement about the growth of $|kA|$ into $d$-dimensions.

\appendix

\section{Corrections to \textquotedblleft Generalized More Sums than Differences"}

In this paper, many of our proof methods followed directly from those in \cite{ILMZ}.
However, in closely going through these authors' work, we found some mistakes which must be corrected, especially as our work builds directly on theirs. 

\begin{itemize}
    \item In Case $1$ in the proof of Theorem $2.2$, $M$ should be defined as  $M=[kr-2k+1-d_1, n-(kr-2k+1-d_1)]$ instead of $M=[kr-2k+1-d_1, kn-(kr-2k+1-d_1)]$.

    \item The proof of Theorem $2.2$ should be split up into the case where $s_1>d_1$ and the case where $s_1=d_1$. The $s_1>d_1$ case gives the original result that $|s_1A-d_1A|=|s_2A-d_2A|+1$, but the $s_1=d_1$ case gives $|s_1A-d_1A|=|s_2A-d_2A|+2.$ 
    
    \item In the proof for Lemma $3.1$, the authors state \textquotedblleft$[2m - 2u, 2m - 2 - 1] \subset U + U$". This should be \textquotedblleft$[2m-2u, 2m - 2 - 2] \subset U + U$."
    
    \item Based on its proof, Theorem $1.4$ only holds for $s_1, d_1\geq 1$, instead of $s_1, d_1\geq 0$ as claimed.
    
    \item In Lemma $4.1$,  equation $(4.3)$ should have $(n-(kr-2k+1-d)-R)$ instead of $(n-(2kr-2k+1-d)-R)$ 
    
    \item For Lemma $5.1$, the authors claim that for $k \geq a_m,$ the middle of the set is filled-in. However this does not always hold. For example for the set $A=\{0,5,8\},$ we have $a_m=8$. Thus for $k=a_m$, $54$ should be in $kA$. However the only way to get $54$ with sums of the elements in $A$ is $6(5)+3(8)=54.$ This would require $k=9$.
    
    \item For Lemma $5.1$, the right fringe $R$ is defined incorrectly. It should be $R=kA \cap [(k-a_1)a_m, ka_m]$, instead of $R=kA \cap [(k-1)a_m, ka_m]$. Then this fringe is symmetric to the left fringe, which the authors' proof needs to be true. 
    
    \item For Lemma $5.1$, a lemma needed to show that there are no $k$-generational sets for all $k$, the authors define $A=\{0, a_1, a_2, \ldots, a_m\} \subset [0, n-1]$ to be a set of integers where $a_1 < a_2 < \cdots < a_m$ and assume $\gcd(a_1, \ldots, a_m)=1$. They then define $B=\{a'_2, \ldots, a'_m\}$ where $a_i=a_i \mod a_1$ and claim that $\gcd(a_2', \ldots, a'_m)=1.$ However we can find a counterexample to the statement $\gcd(a_2', \ldots, a'_m)=1.$    Let $a_1 \in \mathbb{N}$ and $1<x < a_1$ such that $a_1$ and $x$ are coprime. Suppose $A=\{0, a_1, a_2, \ldots, a_m\}$ where $a_1< a_2< \ldots < a_m$  and $a_i=c_{1,i}a_1 +c_{2,i}x$ where $1 \leq c_{1,i}$, $0 \leq c_{2,i}x \leq a_1$. Then for
    if  $B=\{a'_2, \ldots, a'_m\}$ as defined above, we have $\gcd(a_2', \ldots, a'_m)=x$ or $=0.$ To fix this claim, we prove in the following theorem that for $A$, $|A-A| \geq |A+A|$ and therefore if we are studying $k$-generational sets, it doesn't even make sense to consider these sets.
\end{itemize}

\begin{theorem}\label{correction}
Let $a_1 \in \mathbb{N}$ and $1<x < a_1$ such that $a_1$ and $x$ are coprime. Suppose $A=\{0, a_1, a_2, \ldots, a_m\}$ is a set of integers where $a_1< a_2< \ldots < a_m$  and $a_i=c_{1,i}a_1 +c_{2,i}x$ where $1 \leq c_{1,i}$, $0 \leq c_{2,i}x \leq a_1$.  Then $|A-A| \geq |A+A|$.
\end{theorem}

\begin{proof}
To prove this theorem, we first show that $|A-A| \geq |A+A|$ is true in two specific cases. We then show that every perturbation that can create a set with the properties stated in the theorem  still gives us $|A-A| \geq |A+A|$.

If every element greater than $a_1$ is of the form $a_i=c_{1,i}a_1$, we then have an arithmetic progression, which is balanced. 

We then examine the case where every element in $A$ that is not $a_1$ or $0$ is of the form  $a_i=c_{1,i}a_1 +c_{2,i}x$ where $c_{2,i} \neq 0.$ We start with the sub-case where the elements in $A$ that are not $a_1$ or $0$ come from a the intersection of intervals  with arithmetic progressions $A_{\ell}$ of the form $a_i=c_ia_1 +\ell x$ where $1< \ell x<a_i$. In other words, each $A_{\ell}$ is a section of an arithmetic progression that is not missing any points with common difference $a_1$ with each element equal to $\ell x \mod a_1$. Let $B=\bigcup A_{\ell}$ and let $A=\{0, a_1\} \ \cup \ B$. We note that the union of arithmetic progressions is sum-difference balanced, thus we can let  $n=|B-B|=|B+B|.$

We then have $A+A=(B+B) \ \cup \ (0+B) \ \cup \ (0+0)  \ \cup \ (0 + a_1) \ \cup \ (a_1 +a_1) \ \cup \ (a_1 + a_m) $. This can be verified to be a disjoint union as $a_1$ does not divide $a_m$, therefore $|A+A|=n + (m -1) + 1 + 1 +1 +1= n+m+3 $.

Additionally, $A-A=(B-B) \ \cup \ (0-B)  \ \cup \ (B-0) \ \cup \  (0-a_m ) \ \cup \ (a_m-0).$ This too can be verified to be a disjoint union, therefore $|A-A|=n + 2(m -1) + 1 + 1 = n+2m$.

When $m>2$, we have $|A+A|\leq |A-A|.$ We note $m>2$ is a reasonable assumption to make as Hegarty \cite{He} proved that the smallest MSTD set has $8$ elements.

We then see what happens when we perturb the sections of arithmetic progressions that comprise $A$. Suppose the length of an arithmetic progression remains the same, but a \textquotedblleft skip" is introduced into it. For example, if $\{5,8,11, 14\}$ becomes $\{5,11, 14, 17\}$. 

It can be verified that $|A+A|$ grows by 2, and $|A-A|$ grows by 4.
If we add to the length of the skip (still while keeping the number of elements the same), we then have $|A+A|$ grows by $1$ and and $|A-A|$ grows by 2. Thus for $A_{\ell}$ which are subsets of  arithmetic progressions with common difference $a_1$ with each element equal to $\ell x \mod a_1$   and $A=\{0, a_1\} \bigcup A_{\ell}$, we have $|A-A| \geq |A+A|$.

We then see what happens when we introduce points that are equal to $0 \mod{a_1}$ to our set or, in other words, add points of the form $a_i=c_{1,i}a_1.$ We claim that introducing a point of this form either increases the deficit between $|A+A|$ and  $|A-A|$  or keeps it the same. Let $A'$ be $A  \ \cup \ a_n$ where $a_n = 0 \mod{a_1}.$ We then have $A'+A'=(A+A) \ \cup \ (A+a_n)  \ \cup \  (a_n+a_n)$ and $A'-A'=(A-A) \ \cup \ (A-a_n) \ \cup  \ (a_n-A)$. In the extremal case, $(A-a_n)$ and $(a_n-A)$ differ by one point, $an-a_1$  and $a_1-a_n$ respectively, thus $|A'+A'|\leq |A'-A'|.$  One can see that adding in more points that are equal to $0 \mod{a_1}$ has a similar effect. 

Through these perturbations, any set that satisfies the requirements of the theorem can be created. Therefore any set that satisfies the characteristics outlined in the theorem cannot be MSTD.
\end{proof}


\end{document}